\documentclass[oneside,english]{amsart}
\usepackage[T1]{fontenc}
\usepackage[latin9]{inputenc}
\usepackage{color}
\usepackage{amstext}
\usepackage{amsthm}
\usepackage{amssymb}
\usepackage{graphicx}

\makeatletter
\numberwithin{equation}{section}
\numberwithin{figure}{section}
\theoremstyle{plain}
\newtheorem{thm}{\protect\theoremname}[section]
  \theoremstyle{definition}
  \newtheorem{defn}[thm]{\protect\definitionname}
  \theoremstyle{plain}
  \newtheorem{lem}[thm]{\protect\lemmaname}
  \theoremstyle{remark}
  \newtheorem{rem}[thm]{\protect\remarkname}
  \theoremstyle{plain}
  \newtheorem{prop}[thm]{\protect\propositionname}
  \theoremstyle{plain}
  \newtheorem{cor}[thm]{\protect\corollaryname}
  \theoremstyle{definition}
  \newtheorem{example}[thm]{\protect\examplename}

\makeatother

\usepackage{babel}
  \providecommand{\corollaryname}{Corollary}
  \providecommand{\definitionname}{Definition}
  \providecommand{\examplename}{Example}
  \providecommand{\lemmaname}{Lemma}
  \providecommand{\propositionname}{Proposition}
  \providecommand{\remarkname}{Remark}
\providecommand{\theoremname}{Theorem}

\begin{document}

\title{Rank of tropical curves and tropical hypersurfaces}

\author{Boaz elazar$^{\dagger}$}

\address{Department of Mathematics, The Weizmann Institute of Science, Rehovot,
Israel}

\email{boaz.elazar@weizmann.ac.il }

\thanks{$^{\dagger}$The author has been partly supported by the Israeli
Science Foundation grant no. 448/09}
\begin{abstract}
This paper is devoted to the bounding and computation of the dimension
of deformation spaces of tropical curves and hypersurfaces. This characteristic
is interesting in light of the fact that it often coincides with the
dimension of equisingular (equigeneric etc.) deformation spaces of
algebraic curves and hypersurfaces. In this paper, we obtain a series
of precise formulas, upper and lower bounds, and algorithms for computing
dimension of deformation spaces of various classes of tropical curves
and hypersurfaces.
\end{abstract}

\maketitle

\section{Introduction}

The goal of the present work is to study deformation spaces of embedded
and parameterized planar tropical curves, spacial tropical curves
and affine tropical hypersurfaces. The dimesion of a curve's deformation
space is called \textit{the rank} of the curve. The same definition
is used with hypersurfaces. We either give explicit formulas for their
dimension, when possible, or provide lower and upper bounds for it.
In some cases, efficient algorithms for computing the precise values
of the dimension are given.

Our results are related to the definition of end-marked tropical curves
and some results regarding them (Section 3), ranks of plane tropical
curves (Section 4), and ranks of tropical surfaces in $\mathbb{R}^{3}$
(Section 5). We will now shortly describe and comment on the results
of the present work leaving a complete formulation for Sections 3,4
and 5.

(1) \textit{Precise formulas}. We exhibit two types of precise formulas
for the dimension of deformation spaces of tropical curves and hypersurfaces.

One of them equates the actual and the expected dimensions. The latter
means the value obtained by counting the conditions on the parameters
imposed by local combinatorial data. The result is that an equisingular
family of plane tropical curves is always of expected dimension if
the number of higher singularities does not exceed $2$ (Corollary
\ref{cor:param graph equ 1}). Moreover, the actual dimension can
be greater than the expected one already for families of tropical
curves with three higher singularities (Example \ref{exm:3 higher sing}).
It is worth to say that this differs from the algebraic analogue,
which states that an equisingular family of plane algebraic curves
is of expected dimension in one of two cases: First - when the curves
have only nodal singularities {[}15{]} (see a similar statement in
the tropical approach in {[}9, 12{]}). Second - when the number of
higher (non-nodal) singularities does not exceed a bound proportional
to the degree {[}5, 13{]}. We would like to remark that our satatement
(in the tropical approach) can be used for a tropical enumeration
of algebraic curves with any number of nodes and one cusp.

The other precise formula uses some ordering in addition to the local
combinatorial data. It holds in a more general situation: for plane
tropical curves with at most $3$ non-nodal singularities (Theorem
\ref{thm:3 non triang}) and for higher-dimensional tropical hypersurfaces
having at most $3$ singularities (Theorem \ref{thm:n>3 second}).

(2) \textit{Lower and upper bounds}. The \textit{expected rank} is
computed under the condition that all the relations imposed to parameters
are independent, as defined below, following {[}12{]} and {[}9{]}.
The lower bounds to the rank are usually given by the expected rank.
By attaching additional combinatorial information, we strengthen this
lower bound for plane tropical curves (Theorem \ref{thm:some non triang})
and generalize it to higher dimensions (Theorem \ref{thm:n>3 first}).

For the upper bounds we apply another idea. We consider plane tropical
curves with their parameterizations and subdivide the parameterizing
graph into simple (trivalent) pieces. Then we compute the dimensions
of deformation spaces for each component, separately deducing an upper
bound for the rank of the original curve. Furthermore, by analyzing
possible dependence of conditions in the latter consideration, we
provide a precise formula for the rank of an arbitrary plane tropical
curve (Theorem \ref{thm:bounded comp 1}).

At last, we study deformation spaces of spacial tropical curves similar
to 1-dimensional skeleta of tropical surfaces (hypersurfaces). These
tropical curves rather differ from the simple ones studied by Mikhalkin,
since they do not have trivalent vertices. We give an upper bound
for their rank in Theorem \ref{thm:closed volumes}.

(3) \textit{Algorithms}. The results of the theorems are all based
on algebraic or combinatorial calculations, like matrix buildings
(Theorem \ref{thm:bounded comp 1}) and calculations of expected ranks
of surfaces (Theorem \ref{thm:Algo}).

$\:$

\section{Preliminaries}

In This part we describe tropical curves and tropical hypersurfaces
due to two approaches. First, we define them as a combinatorial dual
to some convex polytope with a subdivision. This approach follows
the definitions and results of {[}12{]}. The second uses parameterizes
graphs embedd into $\mathbb{R}^{2}$ with some conditions. This approach
follows the definitions and results of {[}9{]}.

\subsection{Tropical hypersurfaces and subdivisions of Newton polytopes}

Given the non-Archimedean field $\mathbb{K}$ of convergent Puiseux
serieses, define a non-Archimedean valuation map $Val:\mathbb{K}^{*}\to\mathbb{Q}$.
Let $f$ be a Laurent polynomial, namely, a polynomial equation 
\[
f\left(z\right)\equiv\sum\limits _{\omega\in\Delta\cap\mathbb{Z}^{n}}A_{\omega}z^{\omega},
\]
 where $\Delta$ is the Newton polytope (the convex hull of the points
$\omega\in\mathbb{Z}^{n}$ such that $A_{\omega}\neq0$), $A_{\omega}\in\mathbb{K}^{*},\:z^{\omega}=z_{1}^{\omega_{1}}...z_{n}^{\omega_{n}}$
. The tropical polynomial
\[
(1)\qquad N_{f}\left(x\right)=\max\limits _{\omega\in\Delta\cap\mathbb{Z}^{n}}\left(\langle\omega,x\rangle+c_{\omega}\right),
\]
where $x\in\mathbb{R}^{n},\:c_{\omega}=Val\left(A_{\omega}\right)$,
is called the $tropicalization$ of $f$. $N_{f}\left(x\right)$ is
a piecewise linear convex function. 
\begin{defn}
The corner locus of $N_{f}$, i.e. the points where $N_{f}$ is not
linear, defines our object of interest - the \textit{tropical hypersurface}
(see {[}2{]} for details).
\end{defn}
Let $f$ be as before, $I=\left\{ \omega|A_{\omega}\neq0\right\} \subset\mathbb{Z}^{n}$,
and $\Delta=conv\left(I\right)$. Take the convex hull $\hat{\Delta}$
of the set $\left\{ \left(\omega,-Val\left(c_{\omega}\right)\right)\in\mathbb{R}^{n+1}|\omega\in I\right\} $.
Let $\nu_{f}:\Delta\to\mathbb{R},\:\nu_{f}\left(\omega\right)=min\left\{ x|\left(\omega,x\right)\in\hat{\Delta}\right\} $.
This is a convex piecewise linear function. Its linearity domains
subdivide $\Delta$ into convex polytopes with vertices in $I$. Denote
this subdivision by $S_{f}$. 
\begin{lem}
{[}12 - Lemma 2.1{]}, The subdivision $S_{f}$ of $\Delta$ defined
above is combinatorially dual to the corner-locus of $N_{f}$. 
\end{lem}
When we shall use this approach to descrive tropical curves in $\mathbb{R}^{2}$,
we shall use the following definition due to {[}12{]}:
\begin{defn}
The \textit{expected rank} of a tropical curve $T$ obtained as a
dual to a subdivision of a Newton polygon $\Delta$ is $rk_{exp}\left(T\right):=\#Vert\left(S_{T}\right)-1-\sum\limits _{\delta}\left(\#Vert\left(\delta\right)-3\right)$
where $T\subset\mathbb{R}^{2}$ is an embedded plane tropical curve,
$\Delta$ is its Newton polygon and $S_{T}$ is the dual subdivision
of $\Delta$. Here, $Vert\left(S_{T}\right)$ is the set of vertices
of $S_{T}$ and $Vert\left(\delta\right)$ is the set of vertices
of $\delta$, where $\delta$ runs over all polygons of $S_{T}$ .
\end{defn}
\begin{rem}
\label{rem-exp-rank}The intuitive way to understand the expected
rank of a curve is as the sum of all \textit{freedom degrees}, represented
by the vertices of $S_{T}$, minus one, because fixing one such vertex
cannot change the curve. Then, we reduce the result by one for each
\textit{condition} imposed to the curve by each overvalance vertex,
represented by a polytope which is not a simplex. I.e. the ``first''
three vertices of each such polytope represent the intersection of
three planes of the tropical polynomial $N_{f}$, what defines a point.
When we want to set another plane, with set slopes, to fit the point,
its coefficient is not free. It must equate a specific value. But
some conditions may have already taken into account as they are adjacent
to more than one polytope. This is why there is a difference between
the expected rank and the actual rank. Take a look at the next figure
for example.

\includegraphics[scale=0.4]{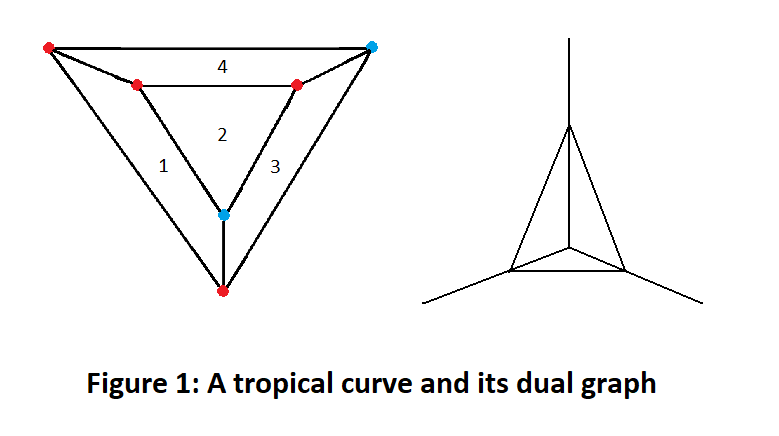}

The red vertices of polytope 1 can be set as we wish. The blue one
has to have a specific parameter or else it will not fit. So the expected
rank is reduced by one. Then the second polytope has only two vertices
set, and we can set the third as we wish. The fourth vertex of the
third polytope is set, and then we get to the fourth polytope. This
polytope shows what we define as a \textit{double condition} - the
blue vertex of it was already set to fit the third polytope, but we
do not know if this setting fits the fourth polytope. If it does,
then there is no further change needed for the rank, but if it does
not, we need to reduce the rank once more, as we need to adjust one
of the other vertices so that the settings will fit. This is why it
is called a double condition. The expected rank assumes the ``worst''
case, so we reduce the expected rank by one for each such vertex.
We shall see the computation in Example \ref{exm:3 higher sing}.
\end{rem}
Let $X\subset\mathbb{R}^{n}$, $n\geq2$, be a tropical hypersurface
with Newton polytope $\Delta$ and its dual subdivision $S_{X}$.
The set $Def^{es}\left(X\right)$ of tropical hypersurfaces in $\mathbb{R}^{n}$
with the Newton polytope $\Delta$ and its dual subdivision coinciding
with $S_{X}$ is called the \textit{equisingular deformations space}
of the hypersurface $X$.
\begin{lem}
({[}12{]}) $Def^{es}(X)$ can be embedded into $\mathbb{R}^{N-1}$,
where $N$ is the number of vertices of the subdivision $S_{T}$,
as a convex polyhedron.
\end{lem}
The parameters of this embedding are the coefficients $c_{\omega}$
of the tropical polynomial (1), where $\omega$ runs over all but
one vertex of the given subdivision of $\Delta$. Each coefficient
is attached to a vertex of the subdivision.
\begin{defn}
\textit{The rank} of a tropical hypersurface is determined as the
dimension of the equisingular deformation space of the hypersurface.
For a tropical hypersurface $X$, we denote its rank by 
\[
rk\left(X\right):=dimDef^{es}\left(X\right).
\]
\end{defn}

\subsection{Tropical curves in $\mathbb{R}^{2}$ as parameterized graphs}

Let $\bar{\Gamma}$ be a weighted, finite and compact graph with weights
in $\mathbb{N}$. Define $\Gamma:\bar{\Gamma}\backslash\nu_{1}$ where
$\nu_{1}$ are the 1-valent vertices. Note that $\Gamma$ is non-compact.
\begin{defn}
\textit{A parameterized tropical curve} is a pair of $\bar{\Gamma}$
without divalent vertices, and a proper map $h:\Gamma\to\mathbb{R}^{n}$
satisfying these two conditions:
\end{defn}
\begin{enumerate}
\item The image $h\left(E\right)$ of an edge $E\subset\Gamma$ is contained
in a line $l\subset\mathbb{R}^{n}$ with a rational slope, such that
$h|_{E}$ is either an embedding or maps $E$ to a point.
\item For every vertex $V\in\Gamma$ the following property holds: let $E_{1},...,E_{m}\subset\Gamma$
be the edges adjacent to $V$ and $\omega_{1},...,\omega_{m}\in\mathbb{N}^{m}$
be their corresponding weights. Let $v_{1},...,v_{m}\in\mathbb{Z}^{n}$
be the primitive integer vectors at the point $h\left(V\right)$,
in the direction of $h\left(E_{1}\right),...,h\left(E_{m}\right)$.
Note, that if $h\left(E_{i}\right)$ is a point for some $i$, then
we take $v_{i}=0$. Then 
\[
\sum\limits _{i=1}^{m}\omega_{i}v_{i}=0.
\]
\end{enumerate}
Two parameterized tropical curves $h:\Gamma\to\mathbb{R}^{n},\:h':\Gamma'\to\mathbb{R}^{n}$
are called \textit{equivalent} if there exists a homeomorphism $\Phi:\Gamma\to\Gamma'$
such that $h=h'\circ\Phi$ and the weights of the edges are kept the
same. Two equivalent parameterized tropical curves will be considered
the same here. We call the image $h\left(\Gamma\right)$ \textit{plane
embedded tropical curve}.

When we shall use this approach to describe tropical curves in $\mathbb{R}^{2}$
we shall use the following definition due to {[}9{]}:
\begin{defn}
The \textit{expected rank} of a parameterized tropical curve $\left(\bar{\Gamma},h\right)$
is $rk_{exp}\left(\bar{\Gamma},h\right):=\#End\left(\Gamma\right)+\left(n-3\right)\left(1-g\right)-\sum\limits _{\nu}\left(mt\left(\nu\right)-3\right)$,
where $\#End\left(\Gamma\right)$ is the number of unbounded edges
of $\Gamma$, $g$ is the genus of $\Gamma$, $\nu$ runs over all
vertices of $\Gamma$ and $mt\left(\nu\right)$ is the valence of
$\nu$ regardless the weights.
\end{defn}
For a parameterized tropical curve $\left(\bar{\Gamma},h\right)$
in $\mathbb{R}^{n}$, we deffine its \textit{equiparametric deformation
space} $Def^{ep}\left(\bar{\Gamma},h\right)$ as the set of parameterized
tropical curves $\left(\bar{\Gamma}',h'\right)$ such that there is
a homeomorphism $\psi:\bar{\Gamma}\rightarrow\bar{\Gamma'}$ such
that $h'\left(\psi\left(e\right)\right)$ is parallel to $h\left(e\right)$
for each edge $e$ of $\Gamma$.
\begin{lem}
({[}3, 9{]}) $Def^{ep}\left(\bar{\Gamma},h\right)$ can be embedded
into $\mathbb{R}^{M+n}$, where $M$ is the number of closed edges
of $\Gamma$, as a convex polyhedron.
\end{lem}
The parameters of this embedding are the coordinates of a fixed vertex
of $\Gamma$ and the length of closed edges of $\Gamma$.
\begin{defn}
The rank of a parameterized tropical curve is defined to be 
\[
rk\left(\bar{\Gamma},h\right):=dimDef^{ep}\left(\bar{\Gamma},h\right)
\]
\end{defn}

\section{\textbf{End-marked plane tropical curves}}
\begin{defn}
\label{def-end-marked-curve}(1) Let $\left(\bar{\Gamma},h\right)$
be an irreducible parameterized plane tropical curve. Let $m\leq\#End\left(\Gamma\right)$,
and $\bar{\gamma}=\left(\gamma_{1},...,\gamma_{m}\right)\subset\Gamma$
an ordered configuration of distinct points lying on the interior
of noncompact edges of $\Gamma$, at most one on each edge. We call
the tuple $\left(\bar{\Gamma},h,\bar{\gamma}\right)$ an \textit{end-marked}
tropical curve, and we say that this curve matches an ordered configuration
of points $\bar{p}=p_{1},...,p_{m}\subset\mathbb{R}^{2}$, if $h\left(\gamma_{i}\right)=p_{i},\:1\leq i\leq m$. 

(2) The same definition can be used when $\bar{\Gamma}$ is a straight
infinte line. In this case we may have one point or two distinct points
lying on it.
\end{defn}
\textbf{Notation:} Denote by $Def_{\bar{p}}^{ep}\left(\bar{\Gamma},h\right)$
the set of those curves $\left(\bar{\Gamma'},h'\right)\in Def^{ep}\left(\bar{\Gamma},h\right)$
which lift up to end-marked curves $\left(\bar{\Gamma'},h',\bar{\gamma'}\right)$
matching the given configuration $p$.

$\:$
\begin{defn}
A parameterized tropical curve in $\mathbb{R}^{2}$ is called \textit{simple}
if all the following holds:
\end{defn}
\begin{enumerate}
\item Each vertex of $\Gamma$ has exactly three adjacent edges.
\item The map $h$ immerses $\Gamma$ into $\mathbb{R}^{2}$.
\item No more then two points in $\Gamma$ are mapped into the same point
in $\mathbb{R}^{2}$.
\item If two distinct points in $\Gamma$ are mapped to the same point in
$\mathbb{R}^{2}$, none of them is a vertex of $\Gamma$.
\end{enumerate}
\begin{prop}
In the above notations, let $\left(\bar{\Gamma},h\right)$ be simple.
If $m=\#\bar{p}<\#End\left(\Gamma\right)$, then 
\[
dimDef_{\bar{p}}^{ep}\left(\bar{\Gamma},h\right)=rk\left(\bar{\Gamma},h\right)-m.
\]
\end{prop}
\begin{proof}
As $h\left(\Gamma\right)$ supports an embedded tropical curve, and
as this curve coincides with the corner locus of $N_{f}$ ({[}2{]}),
we may look at the linear domains of $N_{f}$. Restricting $h\left(\gamma_{i}\right)$
to $p_{i}$ means that the parameter of one of the two linear domains
around $p_{i}$ is also restricted, i.e. $a_{i}p_{i}^{x}+b_{i}p_{i}^{y}+c_{i}=a_{j}p_{i}^{x}+b_{j}p_{i}^{y}+c_{j}$
(where $p_{i}^{x}$ and $p_{i}^{y}$ are the coordinates of $p_{i}$)
and either $c_{i}$ depends on $c_{j}$ or vice versa. This means
$dimDef^{ep}\left(\bar{\Gamma},h\right)$ should be reduced by $1$
for each $i$.
\end{proof}
This means that the relations, i.e. the restrictions on the parameters,
imposed by $m<\#End\left(\Gamma\right)$ points on the non-compact
edges of $\Gamma$ are always independent. In turn, for $m=\#End\left(\Gamma\right)$
the relations are dependent, and we call this dependence a \textit{new
balancing condition}. Before stating it, notice that $\left(\bar{\Gamma},h,\bar{\gamma}\right)$
can be represented as a dual to a subdivision of a Newton polytope.
Hence, we can see it as the corner locus of the tropical polynomial
$N_{f}$ as mentioned above. Thus:
\begin{prop}
\textup{\textcolor{black}{\label{prop-new-balanc-cond}new balancing
condition}}\textcolor{black}{{} }- Let $\left(\bar{\Gamma},h,\bar{\gamma}\right)$
be an end-marked curve with $m=\#\bar{\gamma}=\#End\left(\Gamma\right),\:p_{i}=h\left(\gamma_{i}\right),\:1\le i\le m$.
Let $N_{f}\left(x\right)=\max\limits _{\omega\in\Delta\cap\mathbb{Z}^{n}}\left(\langle\omega,x\rangle+c_{\omega}\right)$
be the tropical polynomial dual to the curve. Let $\omega_{i}\in\Delta$
represent the non-compact linear domains of $N_{f}$, i.e. the vertices
of $\Delta$, and $c_{i}:=c_{\omega_{i}}$ the corresponding parameters.
Then 
\[
(2)\qquad\sum\limits _{i=1}^{m}\left\langle \omega_{i}-\omega_{i+1},p_{i}\right\rangle =0\:,
\]
 both the $p_{i}$'s and the $\omega_{i}$'s are ordered clockwise,
$\omega_{i}$ and $\omega_{i+1}$ are the dual vertices adjacent to
$p_{i}$, and $\omega_{m+1}=\omega_{1}$. Furthermore, if $\left(\bar{\Gamma},h\right)$
is simple, then 
\[
dimDef_{\bar{p}}^{ep}\left(\bar{\Gamma},h\right)=rk\left(\bar{\Gamma},h\right)-m+1\:.
\]
\end{prop}
\begin{proof}
For each $i$ we get 
\[
\left\langle \omega_{i}-\omega_{i+1},p_{i}\right\rangle =c_{i+1}-c_{i}
\]
 For $m=\#\bar{\gamma}=\#End\left(\Gamma\right)$ we get $c_{m+1}=c_{1}$
and thus 
\[
\sum\limits _{i=1}^{m}\left\langle \omega_{i}-\omega_{i+1},p_{i}\right\rangle =0
\]
 If $\left(\bar{\Gamma},h\right)$ is simple, we have no other conditions,
and we can arrange the free parameters of the equations as a matrix,
and see that the last condition is dependent on all the others. Therefore
only $\left(m-1\right)$ should be reduced from the rank.
\end{proof}
\begin{rem}
Proposition \ref{prop-new-balanc-cond} actually says two things.
First, it says that there is a dependent condition on the coefficients
$c_{i}$. This means in the case of a simple curve with $m=\#End\left(\Gamma\right)$
set points $p_{i}$, there is zero freedom degrees to the curve bounded
by them (a simple curve is a tree with a rank of $\#End\left(\Gamma\right)-1$
by the Euler characteristic). Furtheremore, if we would like to change
the position of the points $p_{i}$ while keeping the slopes and vertices,
the new balancing condition tells us there is a condition on the points'
positions. Instead of $2\cdot m$ freedom degrees we have only $2\cdot m-1$,
i.e. after moving $m-1$ points in the plane, the last point can only
be moved along a line. This condition is called a \textit{new balancing
condition}, as mentioned above.

See, for example, Figure 2

\includegraphics[scale=0.4]{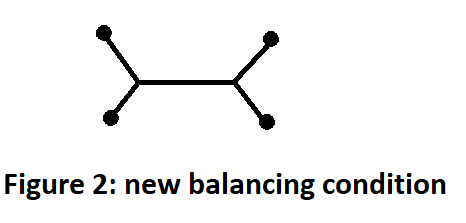}

Here, $dimDef_{\bar{p}}^{ep}\left(\bar{\Gamma},h\right)$ is zero,
because for set points , we cannot chanage anything in the curve between
them. Furthermore, if we want to change the position of the vertices
and keep the same slopes of the edges, and not change the valence
of the vertices, we can change the position of three of them however
we like (in a small range), but the forth can be changed only in one
direction, after the first three were set.
\end{rem}
\begin{defn}
The \textit{expected rank} of an end-marked tropical curve with $m\leq\#End\left(\Gamma\right)$
is $rk_{exp}\left(\bar{\Gamma},h,\bar{\gamma}\right):=rk\left(\bar{\Gamma},h\right)-m$.
\end{defn}
In the case where the end-marked curve is a bounded component, the
expected rank should have been $rk\left(\bar{\Gamma},h\right)-m+1$.
Even though, we keep this definition as is due to the condition to
the vertices' positions, in a manner that will be explained in the
proof of the next theorem.

\section{Rank of plane tropical curves}

In this section we shall study tropical curves in $\mathbb{R}^{2}$.
We shall see the importance of non-nodal vertices; i.e. vertices adjacent
to more than three edges and are not formed by an intersection of
two lines. We shall use a partition of curves by such vertices and
compute ranks to some of the resulting components.
\begin{defn}
Let $T$ be an embedded plane tropical curve. Let $p\in T$ be a vertex
adjacent to 4 or more edges, which is not the intersection of two
edges. We call $p$ a \textbf{non-nodal vertex}.
\end{defn}
$\:$
\begin{defn}
\label{def-bounded-compo}Let $V$ be a set of some vertices of $\Gamma$.
Define a \textit{bounded component} of $\Gamma\backslash V$ to be
a connected component of $\Gamma\backslash V$ such that when adding
the adjacent vertices of $V$, it becomes compact.
\end{defn}
$\:$
\begin{rem}
Let $v^{nn}$ be the non-nodal vertices of $\Gamma$. The embedding
of the connected components of $\Gamma\backslash v^{nn}$, where the
embeddings of the edges adjacent to $v^{nn}$ are extended to infinity,
are called \textit{induced tropical curves}.
\end{rem}
\begin{thm}
\label{thm:bounded comp 1}Let $T$ be an embedded plane tropical
curve and $v^{nn}$ the set of its non-nodal vertices. Replace $\Gamma$
by $\Gamma\backslash h^{-1}\left(v^{nn}\right)$, and let $\left(\bar{\Gamma}_{1},h_{1}\right),...,\left(\bar{\Gamma}_{k},h_{k}\right)$
be the simple parameterizations of the induced (nodal) tropical curces.
Denote by $T_{v^{nn}}$ the set of these induced tropical curves.
Then
\[
rk\left(T\right)\leq rk_{exp}\left(T\right)+\max\left\{ 0,\#\left(\text{bounded components of }\Gamma\backslash h^{-1}\left(v^{nn}\right)\right)-2\right\} .
\]

Furthermore, let M be the matrix of the coefficients of the coordinates
of $v^{nn}$ in the balancing conditions (2) generated by all the
bounded components of $\Gamma\backslash h^{-1}\left(v^{nn}\right)$.
Then
\[
rk\left(T\right)=rk_{exp}\left(T\right)+\#\left(\text{bounded components of }\Gamma\backslash h^{-1}\left(v^{nn}\right)\right)-rkM.
\]
\end{thm}
\begin{proof}
As each $h_{i}\left(\Gamma_{i}\right)\in T_{v^{nn}}$ is nodal, or
non-singular, by {[}12{]}, $rk\left(\bar{\Gamma}_{i},h_{i}\right)=rk_{exp}\left(\bar{\Gamma}_{i},h_{i}\right)$.
Thus, when attaching $\left(\bar{\Gamma}_{i},h_{i}\right)$ to the
relevant vertices of $v^{nn}$ we get an end-marked curve, with rank
$rk_{exp}\left(\bar{\Gamma}_{i},h_{i},End\left(\bar{\Gamma}_{i}\right)\right)=rk_{exp}\left(\bar{\Gamma}_{i},h_{i}\right)-\#End\left(\bar{\Gamma}_{i}\right)$.
As in our case all the induced curves are simple, that have no conditions,
summing the expected ranks of all the induced end-marked tropical
curves, together with the coordinates of $v^{nn}$, yields $rk_{exp}\left(T\right)$.
I.e. $2\cdot\#v^{nn}+\sum rk_{exp}\left(\Gamma_{i}\right)=rk_{exp}\left(T\right)$.
Notice that if an induced end-marked curve is a bounded component,
then the expected rank should have been increased by one due to the
dependence of the coefficients. But the component poses a condition
to the position of the vertices, and thus we still have the same effect,
in case the new balancing condition is independent.

Each of the bounded components of $\Gamma\backslash h^{-1}\left(v^{nn}\right)$
restricts the position of one vertice acording to the new balancing
condition, by Proposition \ref{prop-new-balanc-cond}. The new balancing
conditions may depend on other new balancing conditions. If there
are only one or two bounded components, there is no problem. But,
if there are more, we should get a correction. In order to have a
dependence, there should be at least two independent conditions, and
so we get
\[
rk\left(T\right)\leq rk_{exp}\left(T\right)+\max\left\{ 0,\:\#\left(\text{bounded components of }\Gamma\backslash h^{-1}\left(v^{nn}\right)\right)-2\right\} .
\]

To be exact, we can determine the dependencies of the conditions by
building a matrix $M$ of the coeficients of the coordinates of $v^{nn}$
in the balancing conditions (2) generated by all the bounded components
of $\Gamma\backslash h^{-1}\left(v^{nn}\right)$. The rank of $M$
is exactly the number of independent conditions. So by extracting
the rank of $M$ from the number of bounded components, we get the
number of dependent conditions, which is the exact correction to $rk_{exp}\left(T\right)$.
\end{proof}
For the convenience of the readers, the following lemma is quoted
here, with its proof, for later use such as in Cor. \ref{cor:param graph equ 1}.
\begin{lem}
\label{lem:zero def for 2 overval}({[}12 - Lemma 2.2{]}) 

(1) If $T\subset\mathbb{R}^{2}$ is non-singular or nodal, then 
\[
rk\left(T\right)=rk_{exp}\left(T\right)
\]

(2) If T is singular, and $S_{T}:\Delta=\delta_{1},...,\delta_{N}$
is the dual graph of T divided to its polygons, then for $d\left(S_{T}\right):=rk\left(T\right)-rk_{exp}\left(T\right)$
we get
\[
(2)\:0\leq2d\left(S_{T}\right)\leq\sum\limits _{m\geq2}\left(\left(2m-3\right)N_{2m}-N'_{2m}\right)+\sum\limits _{m\geq2}\left(2m-2\right)N_{2m+1}-1
\]

where $N_{m},\:m\geq3$, is the number of m-gons in $S_{T}$, and
$N'_{2m},\:m\geq2$, is the number of 2m-gons in S, whose opposite
edges are parallel.
\end{lem}
\begin{proof}
(1) The meaning of non-singular $T$, or nodal $T$ is that its dual
graph is built from triangles or parallelograms alone. Therefore,
the conditions imposed by the 4-valent vertices of $T$ are independent,
as we shall prove. Taking a vector $\bar{a}\in\mathbb{R}^{2}$ with
an irrational slope can be used in order to coorient each edge of
any parallelogram in the dual graph so that the normal vector creates
an acute angle with $\bar{a}$. Doing that enables us to make a partial
ordering of the parallelograms, that can be completed into a linear
ordering. Each parallelogram has 2 neighboring edges cooriented outside
of the parallelogram, and 2 that are cooriented inside. Thus the coefficients
of the linear conditions imposed by the 4-valent vertices of $T$
can be arranged into a triangular matrix, meaning they are independent.

(2) If $S_{T}$ contains polygons different from triangles and parallelograms,
we define a linear ordering on the set of all non-triangles in the
same way as in (1). Denote by $e_{-}\left(\delta_{i}\right)$ (respectively,
$e_{+}\left(\delta_{i}\right)$) the number of edges of a polygon
$\delta_{i}$ cooriented outside $\delta_{i}$ (respectively, inside).
Passing inductively over the nontriangular polygons $\delta_{i}$,
each time we add at least $\min\left\{ e_{-}\left(\delta_{i}\right)-1,|Vert\left(\delta_{i}\right)|-3\right\} $
new linear conditions independent of all the preceding ones. Thus,
\[
d\left(S_{T}\right)\leq\sum\limits _{i=2}^{N}\left(|Vert\left(\delta_{i}\right)|-3-\min\left\{ e_{-}\left(\delta_{i}\right)-1,|Vert\left(\delta_{i}\right)|-3\right\} \right)
\]
\[
=\sum_{i=2}^{N}\max\left\{ |Vert\left(\delta_{i}\right)|-e_{-}\left(\delta_{i}\right)-2,0\right\} ,
\]

because, for the initial polygon $\delta_{1}$, all $|Vert(\delta_{1})|-3$
imposed conditions are independent. Replacing $\bar{a}$ by $-\bar{a}$,
we obtain
\[
d\left(S_{T}\right)\leq\sum_{i=1}^{N-1}\max\left\{ |Vert\left(\delta_{i}\right)|-e_{+}\left(\delta_{i}\right)-2,0\right\} .
\]
Since

-) the relations $1\leq e_{-}\left(\delta_{i}\right)\leq|Vert\left(\delta_{i}\right)|-1$
and $e_{-}\left(\delta_{i}\right)+e_{+}\left(\delta_{i}\right)=|Vert\left(\delta_{i}\right)|$
yield $\max\left\{ |Vert\left(\delta_{i}\right)|-e_{-}\left(\delta_{i}\right)-2,0\right\} +\max\left\{ |Vert\left(\delta_{i}\right)|-e_{+}\left(\delta_{i}\right)-2,0\right\} $
$\leq|Vert\left(\delta_{i}\right)|-3$

-) for a $2m$-gon with parallel opposite edges we have $e_{-}\left(\delta_{i}\right)=e_{+}\left(\delta_{i}\right)=m$
\[
\Rightarrow\max\left\{ 2m-e_{-}\left(\delta_{i}\right)-2,0\right\} +\max\left\{ 2m-e_{+}\left(\delta_{i}\right)-2,0\right\} =2m-4,
\]

we get
\[
(3)\:2d\left(S_{T}\right)\leq\sum_{m\geq2}\left(\left(2m-3\right)N_{2m}-N'_{2m}\right)+\sum_{m\geq2}\left(2m-2\right)N_{2m+1}.
\]
\\
If among $\delta_{1},...,\delta_{N}$ there is a polygon $\delta_{i}$
whose number of edges is odd and exceeds 3, or a polygon with an even
number of edges and a pair of nonparallel opposite sides, then $\bar{a}$
can be chosen so that $\min\left\{ e_{-}\left(\delta_{i}\right),e_{+}\left(\delta_{i}\right)\right\} \geq2$.
Thus, the contribution of $\delta_{i}$ to the bound for $2d\left(S\right)$
will be $|Vert\left(\delta_{i}\right)|-4$, which allows us to gain
$-1$ on the right-hand side of formula (3), obtaining formula (2).

Finally, assume that all nontriangular polygons in $S$ have an even
number of edges, that their opposite sides are parallel and that there
is $\delta_{i}$ with $|Vert\left(\delta_{i}\right)|=2m\geq6$. The
union of the finite length edges of $T$ is the adjacency graph of
$\delta_{1},...,\delta_{N}$ . We take the vertex corresponding to
$\delta_{i}$, pick a generic point $O$ in a small neighborhood of
this vertex, and orient each finite length edge of $T$ so that it
forms an acute angle with the radius vector from $O$ to the middle
point of the chosen edge. Equipped with such an orientation, the adjacency
graph has no oriented cycles because the terminal point of any edge
is farther from $O$ than the initial point. Thus, we obtain a partial
ordering on $\delta_{1},...,\delta_{N}$ such that, for any $\delta_{k}$
with an even number of edges, at least half of the edges are cooriented
outside. Then we apply the preceding argument to estimate $d\left(S_{T}\right)$
and notice that the contribution of the initial polygon $\delta_{i}$
to such a bound is zero, whereas on the right-hand side of formula
(3) it is at least two. This completes the proof of formula (2).
\end{proof}
\begin{cor}
\label{cor:param graph equ 1}$\left(1\right)$ If the parameterizing
graph $\Gamma$ of an irreducible parameterized plane tropical curve
$\left(\bar{\Gamma},h\right)$ has at most two vertices of valence
>3, then
\[
rk\left(\bar{\Gamma},h\right)=rk_{exp}\left(\bar{\Gamma},h\right)
\]

$\left(2\right)$ If the dual subdivision $S_{T}$ of the Newton polygon
of an embedded plane tropical curve T contains at most two polygons
other than triangles and parallelograms, then
\[
rk\left(T\right)=rk_{exp}\left(T\right).
\]
\end{cor}
\begin{proof}
(1) Between two vertices of valence >3 there can be only one bounded
component, and therefore the new balancing condition is independent.
(2) According to the proof of lemma \ref{lem:zero def for 2 overval},
if the two polygons do not share an edge, there is no problem. If
they do, it is enough to order the polygons in such a way so that
each of them will have two neighboring edges cooriented outside.
\end{proof}
\begin{example}
\label{exm:3 higher sing}Consider a plane tropical curve $T$ and
its dual subdivision $S_{T}$ with 3 quadrangles that are not parallelograms,
such that $rk(T)>rk_{exp}\left(T\right)$. In this example (cf. Figure
3), the rank of the curve is 3: 2 for transformations, and 1 for multiplication
by a scalar. The expected rank, however, is only 2:
\[
rk_{exp}\left(T\right):=\#Vert\left(S_{T}\right)-1-\sum_{\delta}\left(\#Vert\left(\delta\right)-3\right)=6-1-3=2
\]

\includegraphics[scale=0.4]{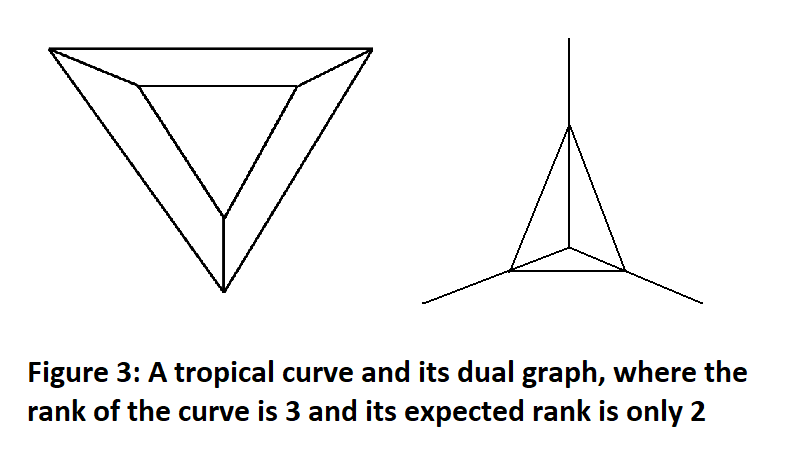}
\end{example}
\begin{thm}
\label{thm:some non triang}Let T be an embedded plane tropical curve,
$S_{T}$ the dual subdivision of its Newton polygon $\Delta$. If
$\delta_{1},...,\delta_{k}$ are all the triangles and parallelograms
in $S_{T}$ and $\delta_{k+1},...,\delta_{l}$ are all the polygons
in $S_{T}$ other than triangles and parallelograms then
\[
rk\left(T\right)\leq\#Vert\left(S_{T}\right)-1-\sum\limits _{\delta\in\left\{ \delta_{1},...,\delta_{k}\right\} }\left(\#Vert\left(\delta\right)-3\right)
\]
\[
-\sum\limits _{i=k+1}^{l}\left(\#Vert\left(\delta_{i}\right)-\max\left\{ 3,\#\left(Vert\left(\delta_{i}\right)\cap\bigcup\limits _{j<i}Vert\left(\delta_{j}\right)\right)\right\} \right),
\]
where $\delta$ runs over the polygons of $S_{T}$.
\end{thm}
\begin{proof}
The number of the parameters we deal with is $\#Vert\left(S_{T}\right)-1$.
For the triangles there are no conditions imposed. The 4-valent vertices
of $S_{T}$, dual to parallelograms, impose independent conditions.
To see these we shall use the method showed in {[}12{]} - take a vector
$\bar{a}\in\mathbb{R}^{2}$ with an irrational slope and coorient
each edge of any parallelogram so that the normal vector forms an
acute angle with $\bar{a}$. This coorientation defines a partial
ordering on the set of parallelograms. We can complete this to a linear
ordering. Each parallelogram has two edges cooriented outside, which
means that the coeffcients of the linear conditions imposed by the
4-valent vertices of $S_{T}$ can be arranged into a triangular matrix,
making the conditions independent. This means that each parallelogram
imposes one condition, and together with the triangles, each one imposes
$\#Vert(\delta_{i})-3$ conditions. For the rest of the polytopes
we can determine only the independence of new conditions, represented
by new vertices which were not part of previous polytopes. The number
of such vertices is
\[
\#Vert\left(\delta_{i}\right)-\#\left(Vert\left(\delta_{i}\right)\cap\bigcup\limits _{j<i}Vert\left(\delta_{j}\right)\right)
\]
for $i\geq k+1$.\\
Some of the \textquotedbl{}old\textquotedbl{} vertices of a polytope
may have been considered as conditions to a former set of equations,
and may as well be considered as conditions in this new set of equations
represented by the polytope. These conditions might not be the same,
so an additional reduction should be made. However, we cannot know
this by this procedure, and this is why we get only an upper bound
for the rank. Changing the partial order may produce other bounds.
The minimum of such upper bounds will give the tightest bound.
\end{proof}
\begin{thm}
\label{thm:3 non triang}If the dual subdivision $S_{T}$ of the Newton
polygon of an embedded plane tropical curve T contains precisely three
polygons $\delta_{1}\:\delta_{2},\:\delta_{3}$ other than triangles
and parallelograms, then
\[
rk\left(T\right)=\#Vert\left(S_{T}\right)-1-\sum\limits _{\delta\neq\delta_{1},\delta_{2},\delta_{3}}\left(\#Vert\left(\delta\right)-3\right)
\]
\[
-\sum\limits _{i=1}^{3}\left(\#Vert\left(\delta_{i}\right)-\max\left\{ 3,\:\#\left(Vert\left(\delta_{i}\right)\cap\bigcup\limits _{j<i}Vert\left(\delta_{j}\right)\right)\right\} \right),
\]
where $\delta$ runs over the polygons of $S_{T}$ .
\end{thm}
\begin{proof}
This theorem follows by Remark \ref{cor-n=00003D3=00003D=00003D>n=00003D2}.
\end{proof}

\section{Rank of tropical surfaces in $\mathbb{R}^{3}$}

As Newton polytopes appear in dimensions higher than 2 in the same
way, we adjust some of the methods of the 2 dimensional case and use
them here. By using the 1-steleton of tropical surfaces, we will adopt
the parameterized methods as well, and give a formula to one such
case.
\begin{thm}
\label{thm:Algo}Lower and upper bounds to the rank of a tropical
surface in $\mathbb{R}^{3}$ can be calculated according to a given
algorithm.
\end{thm}
\begin{proof}
We shall use the dual graph in this case, in order to show an algorithm
to calculate these bounds. We shall denote $\Delta$ to be the Newton
polygon of the tropical surface $A$, and $S_{f}$ its subdivision.

\textit{step 1} : If $A$ has has only one vertex, and $\Delta$ is
a polygon of $m$ planes, then $A$ is an m-valent vertex with faces
between the edges. This means that four of the space's parameters
are independent while all the others are linearly dependent on those
four parameters. This gives the surface four degrees of freedom.

\textit{step 2} : If $\Delta$ has a subdivision $S_{f}$ (e.g. figure
4), one should choose an arbitrary polygon. This can be done as we
deal with locally dependencies. This first polygon, as the one in
step 1, contributes 4 to the rank (e.g. figure 5). Now we shall take
one of its plane neighbors. As they share a plane, three of the degrees
of freedom of the second polygon have already been considered, and
only one more degree of freedom is left to contribute to the rank
(e.g. figure 6). This step sums up to a contribution of five freedom
degrees to the rank.

\includegraphics[scale=0.4]{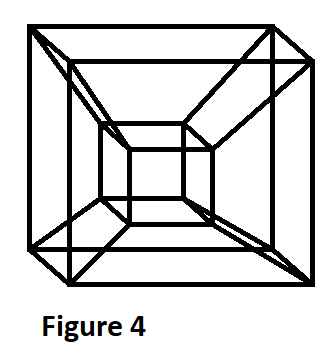}\includegraphics[scale=0.4]{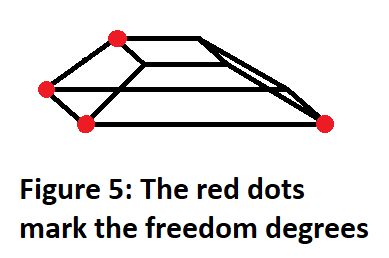}\includegraphics[scale=0.4]{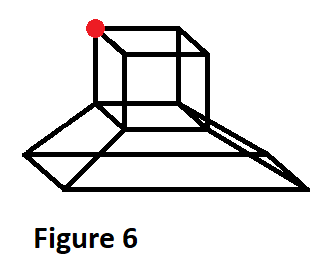}

\textit{step 3} : Now we shall look for polygons which are plane-neighboring
the two previous ones with four vertices or more, where not all of
those vertices lie in the same plane. Those polygons will not add
anything to the rank of the surface, and shall be joined to create
a bloc together with the previous two (e.g. figure 7).

\includegraphics[scale=0.4]{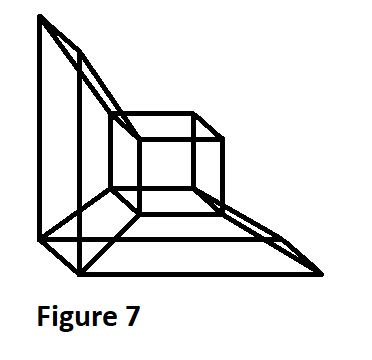}

\textit{step 4} : To the bloc we created on step 3 we can now add
all the polygons which are plane-neighboring the bloc with four or
more vertices that do not lie in the same plane. Again, there will
be no addition to the rank. This step goes on as long as there are
polygons which share with the bloc four or more vertices, as before.
If there are no more polygons of that kind, we go to step 5 (e.g.
figure 4 again).

\textit{step 5} : We look for a polygon plane-neighboring the new
bloc with three vertices. If there is one like that, it shall be added
to the bloc, adding along with it 1 to the rank, as three of its freedom-degrees
have already been considered. Now we shall return to step 4.

\textit{step 6} : If there are no more polygons neighboring the bloc,
we shall reduce 1 from the rank we got. This reduction by one is due
to the isotopy moving the whole graph of the piecewise linear function
\textquotedbl{}up\textquotedbl{} or \textquotedbl{}down\textquotedbl{}
in the four dimensional space, as A is the projection of the corner-locus
of the graph. The only effect on the rank comes from the relationships
between the vertices, and not from their absolute value.

The process is over and we got the upper bound of the rank.

If according to the full order we defined on the polygons, we get
to a polygon whose 5 vertices or more have been already determined,
it means we might have encountered a double condition. A double condition
means a vertex that have to ``fit'' into two or more independent
polytopes of the subdivision, where they might not set the same condition
to that vertex. That means the number of the freedom degrees of the
previous polytope should be reduced so that the conditions on that
vertex will correspond. In such case, $rk\left(A\right)$ should be
decreased by one for every vertex of that kind. As this reduction
is just of suspicious vertices, we get to a lower bound. Changing
the first two polygons in step 2 may produce other bounds. Finally,
the minimum upper bound and the maximum lower bound give the tightest
bounds this procedure allows.
\end{proof}
In figure 4 above, we get that $rk\left(A\right)\leq4$ and the actual
rank is indeed $4$; one for scaling, and three for translations.

$\:$

The next part deals with tropical curves in $\mathbb{R}^{3}$ whose
cycles are planar with no edge going through them, and their vertices'
valence is four or more. The motivation to study these curves is the
1-dimensional skeleta of tropical hypersurfaces in $\mathbb{R}^{3}$.
\begin{rem}
The deformation space of a tropical surface, and the deformation space
of the curve which is derived from the surface as a 1-skeleton, are
the same. The deformations of a face of the tropical hypersurface
are done by deforming its edges. The deformations of 3-dimensional
polytopes are done by deforming their faces, and so on. Therefore,
by induction, the curve which is determined by the intersections of
the polytopes in the tropical hypersurface, actually defines the \textquotedbl{}structure\textquotedbl{}
of the hypersurface and shares the same deformation space with it.

Before our next theorem, some definitions:
\end{rem}
\begin{defn}
\label{def-pre-CL}(1) Let $C$ be a tropical curve in $\mathbb{R}^{3}$
whose minimal cycles are planar and no edge goes through them. Take
only those cycles and close them into faces. If the result consist
a minimal compact polytope, we call the set of the edges of each minimal
compact polytope a \textit{closed volume}. See example in Figure 8.

\includegraphics[scale=0.4]{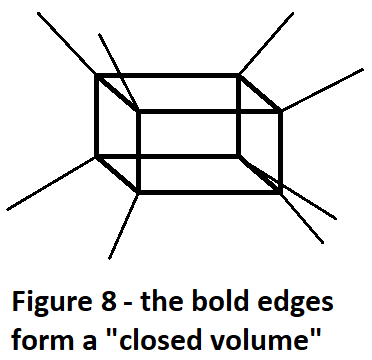}

(2) The \textit{overvalance of a vertex} of a topical curve in $\mathbb{R}^{3}$
is the number of edges adjacent to it minus four. The \textit{overvalance
of a topical curve in $\mathbb{R}^{3}$} is the sum of the overvalaces
of all its vertices. The overvalance of a tropical curve $C$ in $\mathbb{R}^{3}$
is denoted by $ov\left(C\right)$.

(3) Let $T$ be a finite tree of $m$ finite edges, embedded in $\mathbb{R}^{n}$.
We say that $rk\left(T\right)=m+n$.

(4) Let $C$ be a tropical curve in $\mathbb{R}^{3}$ and let $T$
be a maximal tree of $C$. We say $C$ is obtained from $T$ with
the \textit{closed volume property} if the following holds: for each
new closed volume, we have to add an edge and close at least one minimal
cycle in the inductive order of building the curve from the tree.
In other words, there is no closed volume created only by other closed
volumes in the building procedure.
\end{defn}
\begin{thm}
\label{thm:closed volumes}Let C be a tropical curve with valencies
of 4 or more, with no minimal cycles which are not planar, and no
edge goes through those cycles. Let us assume that this curve has
the closed volume property. Then the rank of $C$ is bounded from
below:
\[
rk\left(C\right)\geq\frac{\#End\left(C\right)}{2}+1-\frac{1}{2}ov\left(C\right)+N_{cl}
\]
where $N_{cl}$ is the number of closed volumes.
\end{thm}
\begin{proof}
First, we shall explain $\frac{\#End\left(C\right)}{2}+1-\frac{1}{2}ov\left(C\right)$.
The number of bounded edges in $C$ is $\frac{4\cdot Vert\left(C\right)+ov\left(C\right)-\#End\left(C\right)}{2}$.
In order to get a maximal tree out of $C$, we have to take out $g$
edges, where $g$ is the genus of $C$. Then we get, by the Euler
Characteristic for trees $\left(E=V-1\right)$, that $2\cdot Vert\left(C\right)-\frac{1}{2}\cdot\#End\left(C\right)+\frac{1}{2}ov\left(C\right)-g=Vert\left(C\right)-1$,
which implies that $Vert\left(C\right)=\frac{\#End\left(C\right)-ov\left(C\right)+2g-2}{2}$
and therefore, the number of bounded edged of $C$ from the beginning
of the proof is $\frac{\#End\left(C\right)}{2}+2g-2-\frac{1}{2}\cdot ov\left(C\right)$.
By Definition \ref{def-pre-CL}(3) the rank of the maximal tree $T$
is $rk\left(T\right)=\#\left(\text{bounded edges of }C\right)-g+n$,
where $n$ came from the $n$ dimensional space of translations. Thus,
we get that $rk\left(T\right)=\frac{\#End\left(C\right)}{2}+g-2-\frac{1}{2}\cdot ov\left(C\right)+n$.
As the slope of the edges is fixed, closing each cycle of the curve
from the tree sets a condition on the rank, meaning the length of
one edge of the tree is not free. Since we need $g$ edges in order
to reconstruct the curve from the maximal tree, and since we deal
with $\mathbb{R}^{3}$, we get that $rk\left(C\right)\geq\frac{\#End\left(C\right)}{2}+1-\frac{1}{2}\cdot ov\left(C\right)$
.\\
Regarding $N_{cl}$ - for $N_{cl}=1$: Each of the first $g-1$ minimal
cycles closed from the maximal tree dictates a condition upon the
rank, as explained. Now let us look at the last edge missing in the
closed volume. This last edge is the intersection of two set faces
that goes through two vertices set in this intersection. Therefore,
it does not pose any condition to the curve. For example, remove just
one edge from the closed volume in Figure 8 above. The positions of
all the vertices are set, and the lengths of all the edges are set.
Adding the last edge does not restrict any other edge. Thus, we should
not have subtracted $g$ from the rank of the maximal tree in order
to get to the rank of the curve. Instead, we should have subtracted
only $g-1$, and therefore a correction of one should be added to
the lower bound.\\
For any other closed volume, there is a correction of +1 due to the
closed volume property, which ensures such an edge exists for any
closed volume.
\end{proof}
We should notice that the minimal cycles of the 1-skeleton curve of
a tropical hypersurface in $\mathbb{R}^{3}$ are always like the minimal
cycles in the above deffnition. We can use this fact in order to bound
the rank of this curve by the above theorem. We should also notice
that different hypersurfaces in $\mathbb{R}^{3}$ may have the same
1-skeleton, and therefore their ranks are bounded by the same number.
\begin{rem}
The lower bound for higher dimensional hypersurfaces can be built
in the same manner. For example, 1-skeletons of tropical hypersurfaces
in $\mathbb{R}^{4}$ are at least 5-valent. Taking a hypersurface
in $\mathbb{R}^{4}$ leads to a lower bound according to those considerations:
the number of bounded edges is $\frac{5Vert\left(C\right)-\#End\left(C\right)}{2}$
. The number of vertices is calculated, inductively, by: $\#End\left(C\right)+2g=3Vert\left(C\right)+2+ov\left(C\right)$,
which implies that $Vert\left(C\right)=\frac{\#End\left(C\right)+2g-2-ov\left(C\right)}{3}$
. Combining the two equations leads to $\frac{\#End\left(C\right)}{3}+\frac{5g}{3}-\frac{5}{3}-\frac{ov\left(C\right)}{3}$
bounded edges in the curve and therefore to a maximal tree with rank
$rk\left(Tree\right)=\frac{\#End\left(C\right)}{3}+\frac{2g}{3}-\frac{5}{3}-\frac{ov\left(C\right)}{3}+n$
where $n=4$. Completing the maximal tree into the curve $C$, with
planar cycles, leads to $rk\left(C\right)\geq\frac{\#End\left(C\right)}{3}-\frac{g}{3}+2\frac{1}{3}-\frac{ov\left(C\right)}{3}$
.
\end{rem}
For the next statements we shall use an order on the subdivision,
as proposed in {[}12{]} - take a vector $\bar{a}=\left(a_{1},a_{2},a_{3}\right)\in\mathbb{R}^{3}$
where $\frac{a_{2}}{a_{1}},\:\frac{a_{3}}{a_{1}}$ and $\frac{a_{3}}{a_{2}}$
are irrational. This enables us to define an order on the polytopes
such that other than the first polytope, all the polytopes have at
least one vertex in common with their predecessors. Similarly, any
such order is applicable. Notice that Theorem \ref{thm:n>3 first}
refers higher dimensions as well.
\begin{thm}
\label{thm:n>3 first}Let X be a tropical hypersurface in $\mathbb{R}^{n}$,
$n\geq3$, with the dual subdivision $S_{X}$ of its Newton polytope
$\Delta$. Let $\delta_{1},...,\delta_{k}$ be all the n-polytopes
of $S_{X}$ ordered as stated above. Then
\[
rk\left(X\right)\leqslant\#Vert\left(S_{X}\right)-\sum\limits _{i=1}^{k}\Biggl(\#\biggl(Vert\left(\delta_{i}\right)\backslash\bigcup\limits _{j<i}Vert\left(\delta_{j}\right)\biggl)
\]
\[
-n+dimConvHull\biggl(Vert\left(\delta_{i}\right)\cap\bigcup\limits _{j<i}Vert\left(\delta_{j}\right)\biggl)\Biggl).
\]
\end{thm}
\begin{proof}
The right hand side of the inequality can be presented in the following
way:
\[
\#Vert\left(S_{X}\right)-1+n+1-\#Vert\left(\delta_{1}\right)-
\]
\[
\sum\limits _{i=2}^{k}\Biggl(\#\biggl(Vert\left(\delta_{i}\right)\backslash\bigcup\limits _{j<i}Vert\left(\delta_{j}\right)\biggl)-n+dimConvHull\biggl(Vert\left(\delta_{i}\right)\cap\bigcup\limits _{j<i}Vert\left(\delta_{j}\right)\biggl)\Biggl).
\]

The number of parameters we deal with is $\#Vert\left(S_{X}\right)-1$,
i.e. each vertex denote a hypersurface of dimension $n$ of $N_{f}$,
but setting one hypersurface does not change the corner locus. The
first polytope along the order we have defined has $n+1$ freedom
degrees. All its other vertices represent conditions, meaning these
vertices do not represent freedom degrees of the hypersurface, but
are forced to fit with the previous vertices. Now we continue to the
next polytope. Any new vertex of this polytope is considered as a
new condition. Thus, we reduce the number of the polytope's new vertices
as new conditions. This reduction is noted by
\[
\#\left(Vert\left(\delta_{i}\right)\backslash\bigcup\limits _{j<i}Vert\left(\delta_{j}\right)\right)
\]

A correction to this reduction is needed if the vertices this polytope
shares with its predecessors are positioned in a space with a dimension
less than $n$. A vertex in the tropical hypersurface is defined by
an $n$ dimensional polytope in the dual graph. Thus, some of the
new vertices of the polytope actually do not represent conditions,
and therefore should be added again. The number of such vertices is
exactly that which will make the $n$-dimensional demand. The others
are still considered as conditions. This correction is noted by
\[
-n+dimConvHull\left(Vert\left(\delta_{i}\right)\cap\bigcup\limits _{j<i}Vert\left(\delta_{j}\right)\right).
\]

Take for example the dual graph in Figure 4, and follow by Figure
5,6 and 7.

Some of the \textquotedbl{}old\textquotedbl{} vertices of a polytope
may have been considered as conditions to former set of equations,
and may be considered as conditions on this new set of equations represented
by the polytope. These conditions might not be the same, and therefore,
an additional reduction should be made. However, we cannot know this
by the procedure itself, which is why we get only an upper bound for
the rank. Changing the order may produce other bounds. The minimum
upper bound gives the tightest bound.
\end{proof}
\begin{thm}
\label{thm:n>3 second}Let X be a tropical hypersurface in $\mathbb{R}^{n},\:n=3$,
with the dual subdivision $S_{X}$ of its Newton polytope $\Delta$.
If $S_{X}$ contains at most three n-polytopes $\delta_{i},\:0\leq i\leq i_{0}\leq3$,
other than simplexes, then
\[
rk\left(X\right)=\#Vert\left(S_{X}\right)-\sum\limits _{i=1}^{i_{0}}\Biggl(\#\biggl(Vert\left(\delta_{i}\right)\backslash\bigcup\limits _{j<i}Vert\left(\delta_{j}\right)\biggl)
\]
\[
-n+dimConvHull\biggl(Vert\left(\delta_{i}\right)\cap\bigcup\limits _{j<i}Vert\left(\delta_{j}\right)\biggl)\Biggl)
\]
\end{thm}
\begin{proof}
Each of the higher valence vertices in the tropical graph is dual
to a polytope with five or more vertices in the dual subdivision of
$S_{X}$. Such polytopes shall be later on referred to as \textquotedbl{}OV
polytopes\textquotedbl{}. A problem might arise due to \textquotedbl{}double
conditions\textquotedbl{}, i.e.: conditions that suit to one polytope,
but not to its edge/face neighbor. If no two OV polytopes share vertices,
the rank can be calculated exactly, as conditions appear only on those
three polytopes. If only two share some vertices, the rank can also
be calculated exactly, as the convex hull of the shared vertices is
of a dimension less than that of $\Delta$. Furthermore, we shall
see that even if any two of the OV polytopes share vertices, the rank
can be calculated exactly. If one such pair shares only one vertex,
all the conditions are independent.

The problem arises when each pair shares an edge or a face. The proof
will therefore distinguish between the possible cases in which the
three polytopes are paired, and we will see in each case why the conditions
we encounter in one polytope, also fit to its neighbor. We start with
the cases arise when each pair of OV polytopes share an edge, and
later we continue with cases where one pair or more share a face.

Let us assume, for start, that each pair of the OV polytopes shares
an edge with each of its OV neighbors, i.e.: we have three edges where
each edge belongs to two OV polytopes. Let us refer to those edges
by their boundary vertices $R=\left\{ r_{1},r_{2}\right\} ,\:P=\left\{ p_{1},p_{2}\right\} ,\:Q=\left\{ q_{1},q_{2}\right\} $.
If one of the three edges does not lie in the same plane with one
of the other two edges, then the OV polytopes adjacent to these two
edges do not have a double condition, simply because none of the four
vertices represents a condition. 

Now let us continue to the case where all the three edges are parallel.
An example of such case can be found in Figure 9, where the purple
edges are the shared edges, and the space between the OV polytopes
can be triangulated to simplexes.

\includegraphics[scale=0.4]{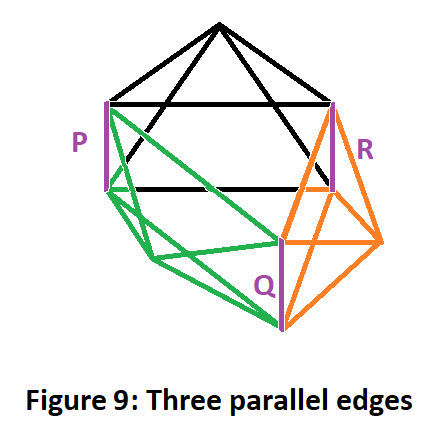}

In that case, by the definition of the tropical polynomial, we have:

\[
(4)\:\left\langle q_{1},x\right\rangle +d_{1}=\left\langle q_{2},x\right\rangle +d_{2}=\left\langle p_{1},x\right\rangle +c_{1}=\left\langle p_{2},x\right\rangle +c_{2}
\]

for some $x\in\mathbb{R}^{3}$ (meaning the four vertices belong to
the same dual polytope),
\[
(5)\:\left\langle q_{1},y\right\rangle +d_{1}=\left\langle q_{2},y\right\rangle +d_{2}=\left\langle r_{1},y\right\rangle +e_{1}=\left\langle r_{2},y\right\rangle +e_{2}
\]

for some $y\in\mathbb{R}^{3}$, and

\[
(6)\:\left\langle p_{1},z\right\rangle +c_{1}=\left\langle p_{2},z\right\rangle +c_{2}=\left\langle r_{1},z\right\rangle +e_{1}
\]

for some $z\in\mathbb{R}^{3}$. We shall now prove that the last equation
can be extended to the following equation:
\[
(7)\:\left\langle p_{1},z\right\rangle +c_{1}=\left\langle p_{2},z\right\rangle +c_{2}=\left\langle r_{1},z\right\rangle +e_{1}=\left\langle r_{2},z\right\rangle +e_{2}
\]

for some $z$. I.e. that the last condition on $e_{2}$ fit as a result
of the previous conditions, and thus does not impose a new condition
and does not ``break'' the third dual polytope into two. 

Reducing (4) from (6) and (4) from (5) gives, respectively:

$\left\langle p_{1},z-x\right\rangle =\left\langle p_{2},z-x\right\rangle \Rightarrow\left\langle p_{1}-p_{2},z-x\right\rangle =0$

$\left\langle q_{1},y-x\right\rangle =\left\langle q_{2},y-x\right\rangle \Rightarrow\left\langle q_{1}-q_{2},y-x\right\rangle =0$

$\Rightarrow\left\langle r_{1}-r_{2},z-y\right\rangle =0$, because
$P||Q||R$. Adding (5) to the last equation leads to (7). Thus, as
$e_{2}$ does not add a new condition, the formula above is exact.

$\:$

Now we shall look at the case where each pair of edges lies in a plane,
but no pair is made of parallel edges (cf. Figure 10), and notice
the figure does not represent a polytope in the dual graph, just the
relations between the shared edges.

\includegraphics[scale=0.4]{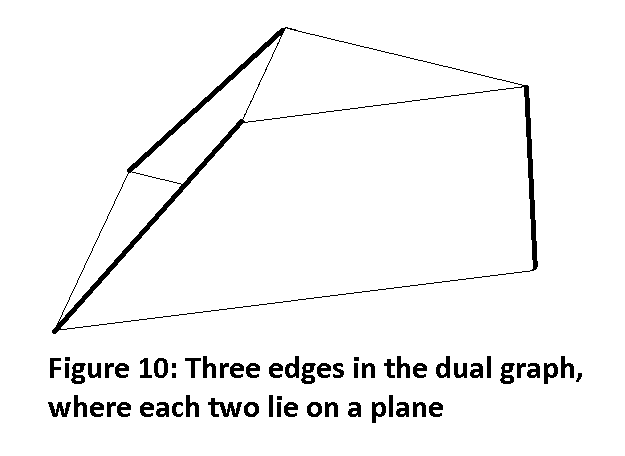}
\[
\]
 Therefore, the constellation of the edges is such that they can be
extrapolated into lines which intersect at a point that we shall call
$"a"$. So we have:\\
 
\[
p_{1}=a+\overline{p},\:p_{2}=a+a'\overline{p},\:1\neq a'\in\mathbb{R}
\]
\[
r_{1}=a+\overline{r},\:r_{2}=a+b'\overline{r},\:1\neq b'\in\mathbb{R}
\]
\[
q_{1}=a+\overline{q},\:q_{2}=a+c'\overline{q},\:1\neq c'\in\mathbb{R}
\]
And we have (4), (5), (6) as before, where we need to prove (7).\\
As we may shift our curve $X$ freely in $\mathbb{R}^{3}$, we can
choose $x=(0,0,0)$ and so we get: 
\[
\left\langle p_{1},(0,0,0)\right\rangle +c_{1}=\left\langle a+\overline{p},(0,0,0)\right\rangle +c_{1}=c_{1}=c_{2}=d_{1}=d_{2}
\]
If we add a constant $k$ to all the monomials of the tropical polynomial
defining $X$, we shall get the same curve $X$. Therefore, we may
choose $c_{1}=c_{2}=d_{1}=d_{2}=0$. So far with $x$.\\
 Now we shall look at $y$. We know that 
\[
\left\langle q_{1},y\right\rangle =\left\langle q_{2},y\right\rangle \Rightarrow\left\langle q_{1}-q_{2},y\right\rangle =0\Rightarrow\left\langle \overline{q},y\right\rangle =0
\]
 (the $d$'s equal to 0). But we also know that $\left\langle q_{1},y\right\rangle =\left\langle a+\overline{q},y\right\rangle =$
$\left\langle a,y\right\rangle +\left\langle \overline{q},y\right\rangle =\left\langle a,y\right\rangle $
and therefore $\left\langle a,y\right\rangle =\left\langle r_{1},y\right\rangle +e_{1}=\left\langle a+\overline{r},y\right\rangle +e_{1}$
$=\left\langle a,y\right\rangle +\left\langle \overline{r},y\right\rangle +e_{1}$,
which implies that\textbf{
\[
e_{1}=-\left\langle \overline{r},y\right\rangle 
\]
}and doing the same with $r_{2}$ and $e_{2}$ leads us to 
\[
e_{2}=-b'\left\langle \overline{r},y\right\rangle .
\]
So far with $y$\textbf{.}\\
Now, let us recall (6) for $z$: $\left\langle p_{1},z\right\rangle =\left\langle p_{2},z\right\rangle =\left\langle r_{1},z\right\rangle +e_{1}$.
And so $\left\langle p_{1},z\right\rangle =$ $\left\langle a+\overline{p},z\right\rangle =\left\langle a,z\right\rangle +\left\langle \overline{p},z\right\rangle $
which equals $\left\langle p_{2},z\right\rangle =\left\langle a+a'\overline{p},z\right\rangle =\left\langle a,z\right\rangle +$
$a'\left\langle \overline{p},z\right\rangle $ and therefore 
\[
(1-a')\left\langle \overline{p},z\right\rangle =0\Rightarrow\left\langle \overline{p},z\right\rangle =0\Rightarrow\left\langle p_{1},z\right\rangle =\left\langle a,z\right\rangle .
\]
This leads to $\left\langle a,z\right\rangle =\left\langle p_{1},z\right\rangle =\left\langle r_{1},z\right\rangle +e_{1}=\left\langle a+\overline{r},z\right\rangle -\left\langle \overline{r},y\right\rangle =$
$\left\langle a,z\right\rangle +\left\langle \overline{r},z-y\right\rangle $,
which implies that $\left\langle \overline{r},z-y\right\rangle =0$\\
 It is now left to prove is that $\left\langle r_{2},z\right\rangle +e_{2}$
equals to (6). In other words, we need to prove $\left\langle r_{2},z\right\rangle +e_{2}=\left\langle a,z\right\rangle $.
The prove is as follows. \\
$\left\langle r_{2},z\right\rangle +e_{2}=\left\langle a+b'\overline{r},z\right\rangle -b'\left\langle \overline{r},y\right\rangle =\left\langle a,z\right\rangle +b'\left\langle \overline{r},z-y\right\rangle =\left\langle a,z\right\rangle $.\\
 \\
The last case is where all the three edges lie in the same plane,
but are not parallel (it is impossible in that case that only two
are parallel, because then it is impossible to have polytopes which
do not intersect each other). Let
\[
p_{1}=a+\bar{p},p_{2}=a+\alpha\bar{p},
\]
\[
q_{1}=a+\bar{q},q_{2}=a+\beta\bar{q},
\]
\[
r_{1}=a+\gamma\bar{q}+\bar{r},r_{2}=a+\gamma\bar{q}+\delta\bar{r}.
\]

Again, we shall set $x=\bar{0}$ and $c_{i}=\text{d}_{i}=0$ and therefore
$\left\langle \bar{q},y\right\rangle =\left\langle \bar{p},z\right\rangle =0$.\\
 We know $\left\langle r_{1},y\right\rangle +e_{1}=\left\langle q_{1},y\right\rangle =\left\langle a,y\right\rangle +\left\langle \bar{q},y\right\rangle $\\
 and also $\left\langle r_{1},y\right\rangle +e_{1}=\left\langle a,y\right\rangle +\gamma\left\langle \bar{q},y\right\rangle +\left\langle \bar{r},y\right\rangle +e_{1}$,
and therefore $e_{1}=-\left\langle \bar{r},y\right\rangle $. In the
same manner we get $e_{2}=-\delta\left\langle \bar{r},y\right\rangle $.
Furthermore, we know 
\[
\left\langle p_{1},z\right\rangle =\left\langle r_{1},z\right\rangle +e_{1}=\left\langle a,z\right\rangle +\gamma\left\langle \bar{q},z\right\rangle +\left\langle \bar{r},z-y\right\rangle 
\]
 and $\left\langle p_{1},z\right\rangle =\left\langle a,z\right\rangle +\left\langle \bar{p},z\right\rangle =\left\langle a,z\right\rangle $,
what leads to $\gamma\left\langle \bar{q},z\right\rangle =-\left\langle \bar{r},z-y\right\rangle $.
As $r_{1}-r_{2}=\left(1-\delta\right)\bar{r}$ forms the edge between
two OV polytopes, which are dual to $y$ and $z$, we get that $\left\langle \bar{r},z-y\right\rangle =0$
due to the characteristics of the duality. As $\left\langle \bar{p},z\right\rangle =0$,
and the three edged are in the same plane but are not parallel, we
get that $\left\langle \bar{q},z\right\rangle \neq0$. Thus $\gamma=0$,
and $\left\langle r_{1},z\right\rangle +e_{1}=\left\langle a,z\right\rangle $.
Now we shall look at $\left\langle r_{2},z\right\rangle +e_{2}=$
$\left\langle a,z\right\rangle +\gamma\left\langle \bar{q},z\right\rangle +\delta\left\langle \bar{r},z-y\right\rangle =\left\langle a,z\right\rangle =\left\langle r_{1},z\right\rangle +e_{1}$.

Notice, that this last case of three edges lying in the same plane,
but are not parallel actually proves Theorem \ref{thm:3 non triang}\textcolor{black}{.
See Remark \ref{cor-n=00003D3=00003D=00003D>n=00003D2}.}

This is the proof for the case where each pair of our three OV polytopes
shares an edge. 

$\:$

Now we shall look at the case where two of the OV polytopes share
a face, spanned by $P=\{p,p+\bar{p},p+\bar{\bar{p}}\}$, and each
other pair shares an edge $Q=\{q_{1},q_{2}\},R=\{r_{1},r_{2}\}$.\textbf{
}We now have a new set of equations: 
\[
(8)\:\left\langle p,x\right\rangle +c_{1}=\left\langle p+\bar{p},x\right\rangle +c_{2}=\left\langle p+\bar{\bar{p}},x\right\rangle +c_{3}=\left\langle q_{1},x\right\rangle +d_{1}=\left\langle q_{2},x\right\rangle +d_{2}
\]
for some $x\in\mathbb{R}^{3}$. 
\[
(9)\:\left\langle q_{1},y\right\rangle +d_{1}=\left\langle q_{2},y\right\rangle +d_{2}=\left\langle r_{1},y\right\rangle +e_{1}=\left\langle r_{2},y\right\rangle +e_{2}
\]
for some $y\in\mathbb{R}^{3}$.\\
 We shall now prove that the following equation 
\[
(10)\:\left\langle p,z\right\rangle +c_{1}=\left\langle p+\bar{p},z\right\rangle +c_{2}=\left\langle p+\bar{\bar{p}},z\right\rangle +c_{3}=\left\langle r_{1},z\right\rangle +e_{1}
\]
for some $z\in\mathbb{R}^{3}$, can be expanded to the following equation:
\[
(11)\:\left\langle p,z\right\rangle +c_{1}=\left\langle p+\bar{p},z\right\rangle +c_{2}=\left\langle p+\bar{\bar{p}},z\right\rangle +c_{3}=\left\langle r_{1},z\right\rangle +e_{1}=\left\langle r_{2},z\right\rangle +e_{2}
\]
for the same $z$.\\
Let us define $\bar{r}$ to be a vector parallel to $r_{2}-r_{1}$,
$\bar{q}$ to be a vector parallel to $q_{2}-q_{1}$, and $P$ the
plane spanned by the vectors $\bar{p}$ and $\bar{\bar{p}}$. There
are several cases possible: the two edges are not in the same plane,
and neither is parallel to $P$ (cf. Figure 11), the edges lie in
the same plane but are not parallel, and neither is parallel to $P$
(cf. Figure 12), the edges are parallel to each other and to $P$
(cf. Figure 13), the edges are parallel but not to $P$ (cf. Figure
14).

\includegraphics[scale=0.4]{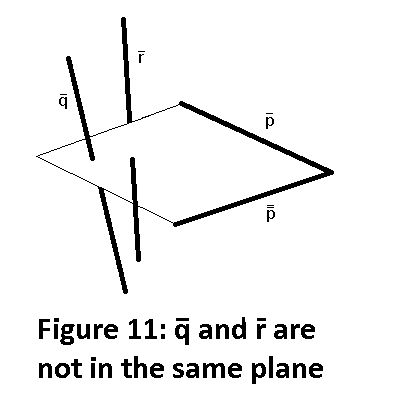}\includegraphics[scale=0.4]{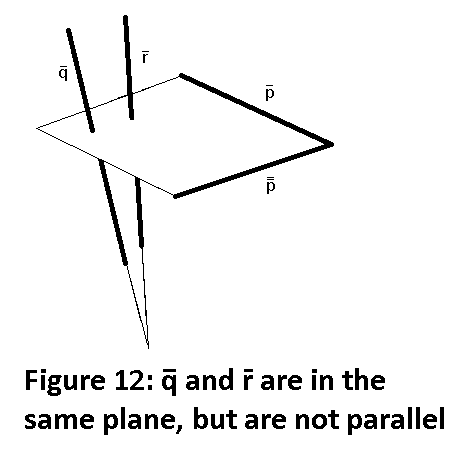}

\includegraphics[scale=0.4]{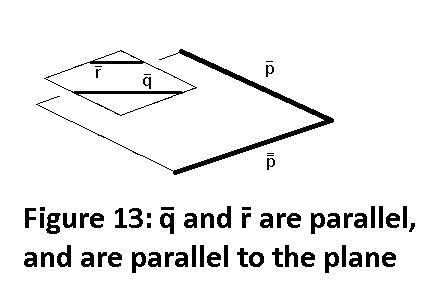}\includegraphics[scale=0.4]{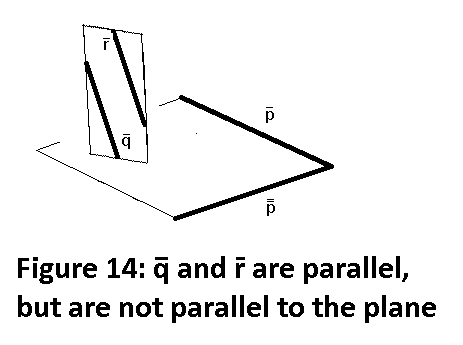}

For all cases we should set $x=\bar{0},\:c_{i}=d_{i}=0$, which implies
that $\left\langle \bar{q},y\right\rangle =$ $\left\langle \bar{p},z\right\rangle =\left\langle \bar{\bar{p}},z\right\rangle =0$.

For the first case we rearrange the equations above, assume equation
(11) holds, and try to extend 
\[
\left\langle q_{1},y\right\rangle +d_{1}=\left\langle q_{2},y\right\rangle +d_{2}=\left\langle r_{1},y\right\rangle +e_{1}
\]

to 
\[
\left\langle q_{1},y\right\rangle +d_{1}=\left\langle q_{2},y\right\rangle +d_{2}=\left\langle r_{1},y\right\rangle +e_{1}=\left\langle r_{2},y\right\rangle +e_{2}
\]

The first case is then obvious, because none of the four vertices
of the polytope dual to $y$ represent a condition to $y$, and thus
no double conditions can occur.\\
From now on, we return to the original arrangement of the equations
(8)-(10) and try to prove (11) follows.

For the second case, where $\bar{r}$ and $\bar{q}$ lie in the same
plane but are not parallel - let us assume $\bar{r}$'s extrapolation
intersects the extrapolation of $P$. Then:\\
 $r_{1}=p+a\bar{p}+b\bar{\bar{p}}+c\bar{r},r_{2}=p+a\bar{p}+b\bar{\bar{p}}+d\bar{r}$\\
 $r_{1}=q_{1}+e\bar{q}+c'\bar{r},r_{2}=q_{1}+d\bar{q}+(c'+d-c)\bar{r}$.\\
 $\left\langle q_{1},y\right\rangle =\left\langle r_{1},y\right\rangle +e_{1}=\left\langle q_{1},y\right\rangle +c'\left\langle \bar{r},y\right\rangle +e_{1}$,
which implies that $e_{1}=-c'\left\langle \bar{r},y\right\rangle $.
In the same manner $e_{2}=-(c'+d-c)\left\langle \bar{r},y\right\rangle $\\
 $\left\langle p,z\right\rangle =\left\langle r_{1},z\right\rangle +e_{1}=\left\langle p,z\right\rangle +c\left\langle \bar{r},z\right\rangle -c'\left\langle \bar{r},y\right\rangle $
which implies that $c\left\langle \bar{r},z\right\rangle =c'\left\langle \bar{r},y\right\rangle $.
Again, as we know there are some $y_{1}$ and $z_{1}$ which fulfil
$\left\langle \bar{r},z_{1}-y_{1}\right\rangle =0$ we get $c=c'$
and so, for all $y$ and $z$ we have $\left\langle \bar{r},z-y\right\rangle =0$.
Therefore $\left\langle r_{2},z\right\rangle +e_{2}=\left\langle p,z\right\rangle +d\left\langle \bar{r},y\right\rangle -(c'+d-c)\left\langle \bar{r},y\right\rangle $
$=$ $\left\langle p,z\right\rangle =\left\langle r_{1},z\right\rangle +e_{1}$.
The meaning of $c=c'$ is that this case is possible only if $\bar{r}$'s
extrapolation meets $\bar{q}$'s extrapolation on $P$'s extrapolation.
Otherwise, this dual division does not represent a hypersurface with
three vertices with valencies greater than four.\\

The third case is that of $\bar{r}\parallel\bar{q}\parallel P$. In
that case $\bar{r}=a\bar{p}+b\bar{\bar{p}}$ and $\bar{q}=a'\bar{p}+b'\bar{\bar{p}}\Leftrightarrow$
$\left\langle \bar{r},z\right\rangle =\left\langle \bar{q},z\right\rangle =0$.
We should also recall that $\left\langle \bar{q},y\right\rangle =0$
and as $\bar{r}\parallel\bar{q}$ we also get $\left\langle \bar{r},y\right\rangle =0$.
Therefore, the question if $\left\langle r_{1},z\right\rangle +e_{1}=\left\langle r_{2},z\right\rangle +e_{2}$
is equivalent to the question whether $e_{1}=e_{2}$, but we know
$\left\langle r_{1},y\right\rangle +e_{1}=\left\langle r_{2},y\right\rangle +e_{2}$
and $\left\langle \bar{r},y\right\rangle =0$, and so we get $e_{1}=e_{2}$.\\
 The last case is where $\bar{r}\parallel\bar{q}$, but $R$'s extrapolation
intersects $P$'s extrapolation. $\left\langle \bar{q},y\right\rangle =0$
and $R\parallel Q$, which implies that $\left\langle \bar{r},y\right\rangle =0$.
$r_{1}=p+a\bar{p}+b\bar{\bar{p}}+c\bar{r}=q_{1}+a'\bar{p}+b'\bar{\bar{p}}+c'\bar{r},r_{2}=p+a\bar{p}+b\bar{\bar{p}}+d\bar{r}=q_{1}+a'\bar{p}+b'\bar{\bar{p}}+(c'+d-c)\bar{r}$.
Let us mark $\widetilde{p}\equiv a'\bar{p}+b'\bar{\bar{p}}$. Now,
$\left\langle q_{1},y\right\rangle =\left\langle r_{1},y\right\rangle +e_{1}=\left\langle q_{1},y\right\rangle +\left\langle \widetilde{p},y\right\rangle +e_{1}$,
which implies that $e_{1}=-\left\langle \widetilde{p},y\right\rangle $,
and in the same manner $e_{2}=-\left\langle \widetilde{p},y\right\rangle $
(as $\left\langle \bar{r},y\right\rangle =0$).\\
 $\left\langle p,z\right\rangle =\left\langle r_{1},z\right\rangle +e_{1}=\left\langle p,z\right\rangle +c\left\langle \bar{r},z\right\rangle -\left\langle \widetilde{p},y\right\rangle $\\
 $\left\langle r_{2},z\right\rangle +e_{2}=\left\langle p,z\right\rangle +d\left\langle \bar{r},z\right\rangle -\left\langle \widetilde{p},y\right\rangle $\\
We know $c\neq d$ and therefore an equality between $\left\langle r_{1},z\right\rangle +e_{1}$
and $\left\langle r_{2},z\right\rangle +e_{2}$ can only be achieved
if $\left\langle \bar{r},z\right\rangle =0$. However, we know this
cannot be since $R$'s and $P$'s extrapolations intersect, and $P$'s
extrapolation is perpendicular to $z$. This means there cannot be
a hypersurface with three vertices with OV which is dual to such dual
subdivision. This close the section of a face and two edges.\\
 $\:$

The next case is of three OV polytopes where two of them share an
edge $R=\{r_{1},r_{2}\}$ and each of the other pairs shares a face
spanned by $P=\{p,p+\bar{p},p+\bar{\bar{p}}\}$ and $Q=\{q,q+\bar{q},q+\bar{\bar{q}}\}$.
We now have a new set of equations: 

\[
(12)\:\left\langle p,x\right\rangle +c_{1}=\left\langle p+\bar{p},x\right\rangle +c_{2}=\left\langle p+\bar{\bar{p}},x\right\rangle +c_{3}=
\]
\[
\left\langle q,x\right\rangle +d_{1}=\left\langle q+\bar{q},x\right\rangle +d_{2}=\left\langle q+\bar{\bar{q}},x\right\rangle +d_{3}
\]
for some $x\in\mathbb{R}^{3}$.\textbf{
\[
(13)\:\left\langle q,y\right\rangle +d_{1}=\left\langle q+\bar{q},y\right\rangle +d_{2}=\left\langle q+\bar{\bar{q}},y\right\rangle +d_{3}=\left\langle r_{1},y\right\rangle +e_{1}=\left\langle r_{2},y\right\rangle +e_{2}
\]
}for some $y\in\mathbb{R}^{3}$.\\
 We shall now prove that the following equation

\[
(14)\:\left\langle p,z\right\rangle +c_{1}=\left\langle p+\bar{p},z\right\rangle +c_{2}=\left\langle p+\bar{\bar{p}},z\right\rangle +c_{3}=\left\langle r_{1},z\right\rangle +e_{1}
\]
for some $z\in\mathbb{R}^{3}$, can be expanded to the following equation:
\[
(15)\:\left\langle p,z\right\rangle +c_{1}=\left\langle p+\bar{p},z\right\rangle +c_{2}=\left\langle p+\bar{\bar{p}},z\right\rangle +c_{3}=\left\langle r_{1},z\right\rangle +e_{1}=\left\langle r_{2},z\right\rangle +e_{2}
\]
for the same $z$.\\
 Let us define $\bar{r}$ to be a vector which is parallel to $r_{2}-r_{1}$.
The possibilities are: The edge and the two faces are all parallel
(cf. Figure 15), the extrapolation of the edge intersects the extrapolation
of the two parallel faces (cf. Figure 16), the extrapolation of the
edge is parallel to the intersection of the extrapolations of the
two faces (which are not parallel) (cf. Figure 17), the edge is parallel
to one face but not to the other (cf. Figure 18). 

\includegraphics[scale=0.4]{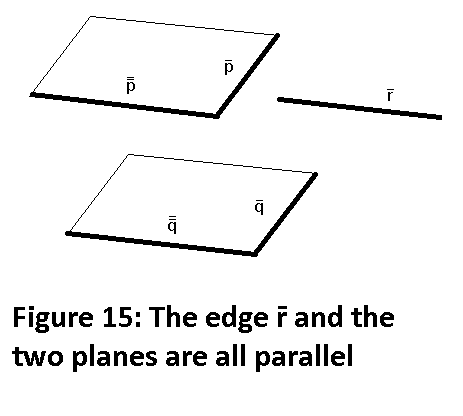}\includegraphics[scale=0.4]{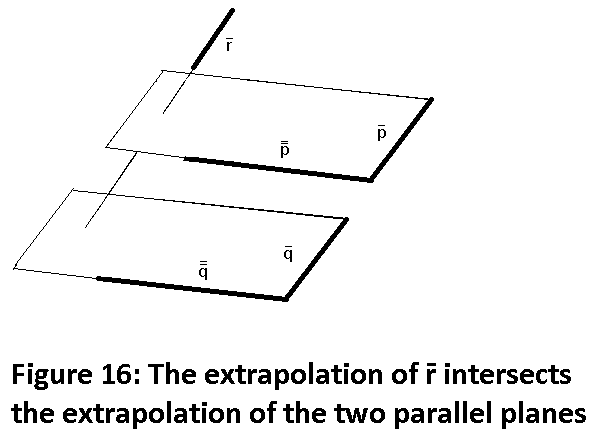}

\includegraphics[scale=0.4]{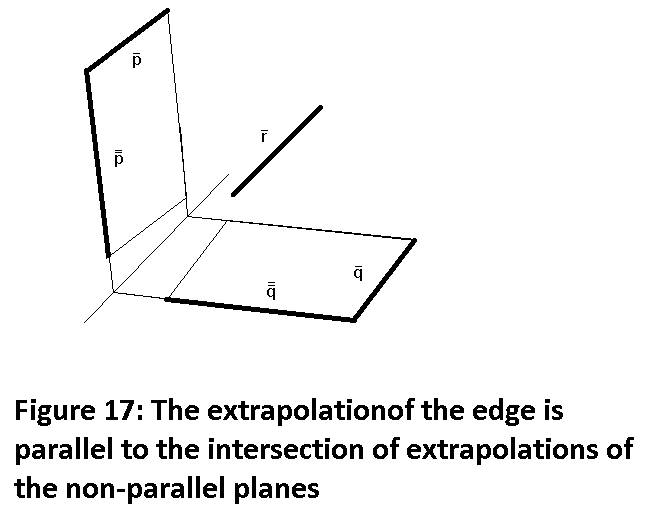}\includegraphics[scale=0.4]{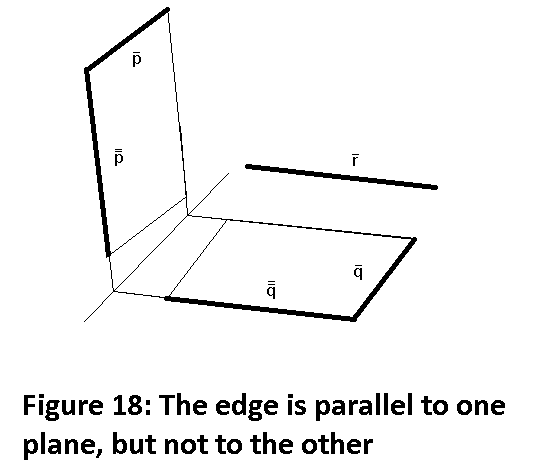}

$\:$

For all the cases we shall take $x=\bar{0},c_{i}=d_{i}=0$, which
implies that $\left\langle \bar{q},y\right\rangle =$ $\left\langle \bar{\bar{q}},y\right\rangle =\left\langle \bar{p},z\right\rangle =\left\langle \bar{\bar{p}},z\right\rangle =0$.
The first case cannot be, because one cannot create non-intersecting
convex polytopes from this constellation. In the second case we have
$r_{1}=q+a\bar{q}+b\bar{\bar{q}}+c\bar{r},r_{2}=q+a\bar{q}+b\bar{\bar{q}}+d\bar{r}$
and so $\left\langle q,y\right\rangle =\left\langle r_{1},y\right\rangle +e_{1}=$
$\left\langle q,y\right\rangle +a\left\langle \bar{q},y\right\rangle +b\left\langle \bar{\bar{q}},y\right\rangle +c\left\langle \bar{r},y\right\rangle +e_{1}=\left\langle q,y\right\rangle +c\left\langle \bar{r},y\right\rangle +e_{1}$,
which implies that $e_{1}=-c\left\langle \bar{r},y\right\rangle $
and in the same manner $e_{2}=-d\left\langle \bar{r},y\right\rangle $.\\
 We can also represent $r_{1}=p+a'\bar{p}+b'\bar{\bar{p}}+c'\bar{r},r_{2}=p+a'\bar{p}+b'\bar{\bar{p}}+(c'+d-c)\bar{r}$
and so $\left\langle p,z\right\rangle =\left\langle r_{1},z\right\rangle +e_{1}=\left\langle p,z\right\rangle +a'\left\langle \bar{p},z\right\rangle +$
$b'\left\langle \bar{\bar{p}},z\right\rangle $ $+c'\left\langle \bar{r},z\right\rangle -$
$c\left\langle \bar{r},y\right\rangle =\left\langle p,z\right\rangle +c'\left\langle \bar{r},z\right\rangle -c\left\langle \bar{r},y\right\rangle $
and as we know, there are $y_{1}$ and $z_{1}$ for which $\left\langle \bar{r},z_{1}-y_{1}\right\rangle =0$
due to the duality, and therefore $c=c'$. That leads to $\left\langle \bar{r},z-y\right\rangle =0$
for all relevant $y$ and $z$. That, in turn, leads to $\left\langle r_{2},z\right\rangle +e_{2}=\left\langle p,z\right\rangle +d\left\langle \bar{r},z\right\rangle -d\left\langle \bar{r},y\right\rangle $
and since $\left\langle \bar{r},z-y\right\rangle =0$ we get $\left\langle r_{2},z\right\rangle +e_{2}=\left\langle p,z\right\rangle =\left\langle r_{1},z\right\rangle +e_{1}$.
The fact that $c=c'$ means that the only option for that case to
describe the wanted hypersurface is if the extrapolation of $R$ intersects
the intersection of the two faces. \\
In the third case we have $\bar{r}=a\bar{p}+b\bar{\bar{p}}=a'\bar{q}+b'\bar{\bar{q}}\Rightarrow\left\langle \bar{r},y\right\rangle =\left\langle \bar{r},z\right\rangle =0$.
Since we know $\left\langle r_{1},y\right\rangle +e_{1}=\left\langle r_{2},y\right\rangle +e_{2}$
we get $e_{1}=e_{2}$. And that is why $\left\langle r_{1},z\right\rangle +e_{1}=\left\langle r_{2},z\right\rangle +e_{2}$.\\
 The last case is the one where $R$ is parallel to $P$, without
loss of generality, and $R$'s extrapolation intersects $Q$'s. Then
we have $\bar{r}=a\bar{p}+b\bar{\bar{p}}$, which implies $\left\langle \bar{r},z\right\rangle =0$.
Since $r_{1}=q+a\bar{q}+b\bar{\bar{q}}+c\bar{r},r_{2}=q+a\bar{q}+b\bar{\bar{q}}+d\bar{r}$
we get $\left\langle q,y\right\rangle =\left\langle r_{1},y\right\rangle +e_{1}=$
$\left\langle q,y\right\rangle +a\left\langle \bar{q},y\right\rangle +b\left\langle \bar{\bar{q}},y\right\rangle +c\left\langle \bar{r},y\right\rangle +e_{1}=\left\langle q,y\right\rangle +e_{1}$,
which implies that $e_{1}=-c\left\langle \bar{r},y\right\rangle $
and the same way leads to $e_{2}=-d\left\langle \bar{r},y\right\rangle $.
Since we know $\left\langle \bar{r},z\right\rangle =0$ the question
whether $\left\langle r_{1},z\right\rangle +e_{1}=\left\langle r_{2},z\right\rangle +e_{2}$
is equivalent to the question whether $e_{1}=e_{2}$ for which we
know the answer is \textquotedbl{}no\textquotedbl{} ,since $c\neq d$
and since $\left\langle \bar{r},y\right\rangle \neq0$. The last inequality
arises from the given fact that $R$ is not parallel to $Q$, and
from the fact that $y$ is perpendicular to $Q$.\\
$\:$

The last option, where each two polytopes intersect in a face, can
be understood from the previous option - a hypersurface as needed
can be created only when all the extrapolations to the faces intersect
in the same line. In this case, there are no double conditions, and
the counting in the order suggested is accurate.
\end{proof}
\begin{rem}
\label{cor-n=00003D3=00003D=00003D>n=00003D2}The proof of the case
where each two polytopes share an edge, and all these edges lie in
the same plane but are not parallel, can actually prove \textcolor{black}{Theorem
\ref{thm:3 non triang}.} Recall that \textcolor{black}{Theorem \ref{thm:3 non triang}}
deals with a planar curve with only three vertices of valency higher
than three in $\mathbb{R}^{2}$, and states that it gets an exact
rank.\textcolor{black}{{} To see that, take the planar graph of Theorem
\ref{thm:3 non triang}, add a vertex not in the plane of the graph,
attach all the other vertices to it and complete the faces as needed.
Now we can apply Theorem \ref{thm:n>3 second} to the result graph. }
\end{rem}

\thanks{\textbf{Acknowledgements:} This work was carried out in Tel-Aviv
University, under the supervision of the supportive and helpful Prof.
Eugenii Shustin, to whom I owe the opportunity to experience this
field of mathematics.\\
I would also like to thank Daniel Benarroch for reading the manuscript
and for his useful remarks. \\
}

\end{document}